\documentclass[12pt]{amsart}
\usepackage{amssymb,latexsym}
\usepackage{enumerate}
\usepackage{mathrsfs}
\DeclareMathOperator{\co}{co}
\DeclareMathOperator{\cco}{\overline{co}}
\DeclareMathOperator{\diam}{diam}
\DeclareMathOperator{\supp}{supp}

\makeatletter
\@namedef{subjclassname@2010}{%
  \textup{2010} Mathematics Subject Classification}
\makeatother

\newtheorem{theorem}{Theorem}[section]
\newtheorem{thm}[theorem]{Theorem}

\newtheorem{lem}[theorem]{Lemma}

\newtheorem{proposition}[theorem]{Proposition}
\newtheorem{prop}[theorem]{Proposition}
\newtheorem*{thrm}{Theorem}

\theoremstyle{definition}

\newtheorem{dfn}[theorem]{Definition}

\newtheorem*{rem}{Remark}
\newtheorem*{ntt}{Notation}

\numberwithin{equation}{section}

\begin{document}
\baselineskip=17pt
\title{On the structure of non-dentable subsets of 
$C(\omega ^{\omega^{{k}}} )$}
\thanks{This is part of the first author's Ph.D thesis, which is in preparation at
the Technical University of Crete under the supervision of the second
author}
\author[P.D. Pavlakos]{Pericles D. Pavlakos}
\address{Technical University of Crete\\
Department of Sciences,
Section  of Mathematics,
\\
73100 Chania,Greece
}\email{pericles@science.tuc.gr}
\email{minos@science.tuc.gr}
\author[M. Petrakis]{Minos  Petrakis}

\dedicatory{Dedicated to the memory of J.J. Uhl, Jr.}
\date{}
\begin{abstract}
It is shown that there is no $K$ closed convex bounded non-dentable subset of $%
C(\omega ^{\omega ^{k}})$ such that on the subsets of $K$ the PCP
and the RNP are equivalent properties. Then applying
Schachermayer-Rosenthal theorem, we conclude that every
non-dentable $K$ contains non-dentable subset $L$ so that on $L$
the weak topology coincides with the norm one. It follows from
known results that the RNP and the KMP are equivalent properties
on the subsets of $C(\omega ^{\omega ^{k}})$.
\end{abstract}



\maketitle
\section*{Introduction}
The study of subsets of Banach spaces with the RNP flourished in the 70's
and 80's. For the history of the subject and results until 1977 one can see
\cite{14}. For more recent results (until 1983), in the form of a book, see \cite{13}.
Important results can be found in \cite{9}. See also \cite{5} for a concise exposition
of RNP.

Many mathematicians worked on the Radon-Nikodym property including: R. Phelps, R. C. James,
J. Diestel, J. J. Uhl, Jr., M. Talagrand, C. Stegall, J. Bourgain, H.
Rosenthal, W. Schachermayer, N. Ghoussoub, B. Maurey, G. Godefroy, S.
Argyros.

Important papers in the field are: \cite{7}, \cite{9}, \cite{10}, \cite{24}, \cite{15}, \cite{22}, \cite{19}.

Cornerstone in our considerations in this paper is the paper
\cite{2} which can be considered as a localization of the results
in \cite{7} and a unification of Bourgain's and Schachermayer's
theorems (\cite{7}, \cite{22}).

Also \cite{4} played an important role in the constructions of bushes in some of
our theorems.

According to \cite{19} the study of the structure of non-dentable sets of a
Banach space is central in the Geometry of Banach spaces.

The Diestel conjecture remains open since 1973 both globally and locally: Is
the KMP equivalent to the RNP?

The most signifigant results related to this problem are the following:

\begin{thrm}(Schachermayer, \cite{22}, Th. 2.1)):If a convex, bounded,
closed subset $D\subset X$ is strongly regular and fails to be an
RN-set, then there is a closed, bounded, convex and separable subset $C$
\textit{\ of }$D$ which does not have an extreme point.
\end{thrm}
\begin{thrm}(Rosenthal, \cite{19}, Th. 2): Let $K$  be a
closed bounded nonempty convex subset of $X$ so that $K$
is non-dentable and has SCS. Then there exists a closed convex nonempty
subset $W$ of $K$ so that 

\emph{(*)} $W$  is non dentable and the weak and norm
topologies on $W$ coincide.

Moreover there exists a subspace $Y$ of $X$ so
that $Y$  has an FDD and a closed bounded convex non-empty subset $
W$  satisfying \emph{(*)}.
\end{thrm}
It is known that for certain classes of spaces (sets) the RNP is equivalent
with the KMP :Dual spaces (Huff-Morris \cite{16}, based on the work of Stegall
\cite{24}), subsets of the positive cone of $L^{1}$ (Argyros-Deliyanni \cite{2}),
spaces which can be embedded to a space with unconditional FDD (James \cite{17}),
spaces with $X\equiv X\oplus X$ (Schachermayer \cite{23}), Banach Lattices
(Bourgain-Talagrand \cite{12}).

It is shown in \cite{2} that in many of the above cases any convex, closed,
bounded non-dentable set contains a subset with the $\mathcal{P}al$
representation.

We believe that a positive answer to the problem of equivalence of the RNP
and the KMP on the closed convex bounded (c.c.b.) subsets of $C(a)$, where $
a $ is a countable ordinal and a similarly positive answer on the c.c.b.
subsets of $L^{1}$, is a strong indication that the RNP and KMP are
equivalent properties on the c.c.b. subsets of a general Banach space $X$.

In this paper we show that the RNP and the KMP are equivalent on the closed
convex bounded subsets of $C(a)$ for ordinals $a<\omega ^{\omega ^{\omega
}}. $

The main results in our paper are:

\emph{Theorem 3.2}: \textit{Let} $X$ \textit{be a
separable Banach space that contains no copy of}
$l_{1}(\mathbb{N}
)$ \textit{and }$Q_{n}:X\rightarrow C(\omega ^{\omega ^{k}})$,$n\in
\mathbb{N} 
$ \textit{be bounded linear operators. Suppose $K$ is a closed,
convex, bounded, non-PCP subset of $X$ , such  that the PCP and the
RNP are equivalent properties on the subsets of $K$. Then  there
exists $L$ closed, convex, bounded, non-dentable subset of $K$,
such that on $Q_{n}(L)$  norm and weak topologies
coincide for all $n\in
\mathbb{N}
$}.

\emph{Theorem 3.3 }:\textit{Let }$K$\textit{\ be a  closed, convex, bounded,
non-dentable subset of }$C(\omega ^{\omega ^{k}})$\textit{. Then there
exists a convex, closed subset }$L~$\textit{of }$K$,\textit{\ such that }$L$
\textit{\ has the PCP and fails the RNP. Therefore the KMP and the RNP are
equivalent on the subsets of }$C(\omega ^{\omega ^{k}})$\textit{.}

The set $L$, mentioned in  Theorem 3.2  is constructed to be the closed
convex hull of a $\delta $ -approximate bush which has the Convex
Finite-Dimensional Schauder Decomposition (C.F.D.S.D.) (\cite{7}, \cite{2}, \cite{21}), and
is given by the closed convex hull of the average back bush of a $\delta -$
approximate bush.

The result of the Theorem 3.3 is the "best" possible concerning the spaces $
C(\omega ^{\omega ^{a}}),$ for $a$ ordinal, since E. Odell \cite{18} has proved,
in unpublished work, that the space $C(\omega ^{\omega ^{\omega }})$
contains a convex, closed, bounded non-dentable subset $L$ where the PCP is
equivalent with the RNP.
\section*{Preliminaries}
\textit{RNP and related properties.}

Let $K$ a closed, convex, bounded subset of a Banach space $X.$

The set $K$ has the \textit{Radon-Nikodym property} (RNP) if for every
probability space ($\Omega ,\mathcal{B},\mu $) and every $X-$valued measure $
m$ on $\mathcal{B}$ which is absolutely continuous with respect to $\mu $
and whose average range is contained in $K,$ there exists an $f\in
L_{X}^{1}(\Omega ,\mathcal{B},\mu )$ such that $m(A)=\int\limits_{A}fd\mu $
(Bochner integral) for each $A\in \mathcal{B}$. It has the 
\textit{Krein-Milman Property} (KMP) if each closed, convex, bounded subset of $K$
is the closed convex hull of its extreme points.

A \textit{slice }$S(f,a,K)$ of $K,$ determined by $f\in X^{\ast }$ and $a>0,
$ is the set $S(f,a,K)=\{x\in K:f(x)\geq \sup f(K)-a\}.$

The set $K$ is said to be \textit{strongly regular} if for every non-empty
subset $L$ of $K$ and any $\varepsilon >0$, there exists positive scalars $%
a_{1},a_{2},...a_{n}$ with $\sum\limits_{i=1}^{n}a_{i}=1$ and slices $%
S_{1},S_{2},...S_{n}$ of $L$ such that the diameter of $\sum%
\limits_{i=1}^{n}a_{i}S_{i}$ is less than $\varepsilon $. The set $K$ has
the \textit{Point of Continuity Property} (PCP) if for every weakly closed
non-empty subset $L$ of $K$ the identity map $i:(L,w)\rightarrow
(L,\left\Vert .\right\Vert )$ has a point of continuity. The set $K$ has the
\textit{Convex PCP }(CPCP)\textit{\ }if for every closed convex non-empty
subset $L$ of $K$ the identity map $i:(L,w)\rightarrow (L,\left\Vert
.\right\Vert )$ has a point of continuity (\cite{15}).

If $K$ is non-PCP then there exists an $L\subseteq K$ and $\delta >0$ so
that $L$ is $\delta -$non-PCP (i.e. for every weak open subset $W$ of $L$ we
have $diamW>\delta $ \cite{7}). Of course if $K$ is $\delta -$non-PCP, then $K$ is
non-PCP.

It is well known that if $K$ has PCP then $K$ is strongly regular (\cite{9}).

\smallskip \textit{Operators on }$L^{1}$ \textit{and RNP.}

Let $\mathcal{P}(\mu )=\{f\in L^{1}(\mu ):f\geq 0$ and $\int fd\mu =1\}$ be
the probability densities in $L^{1}(\mu ).$

It is well known that $K$ has RNP if and only if, every bounded linear
operator $T:L^{1}(\mu )\rightarrow X$ such that $Tf\in K$ for every $f\in
\mathcal{P}(\mu )$ is representable. An operator $T:L^{1}(\mu )\rightarrow X$
is said to be \textit{representable} if there is a function $g\in
L_{X}^{\infty }(\mu )$ such that $Tf=\int fgd\mu $ (Bochner integral) for
every $f\in L^{1}(\mu )$ (\cite{14}).

The set $K$ is strongly regular if and only if every bounded linear operator
$T:L^{1}(\mu )\rightarrow X$ with $T(\mathcal{P})\subseteq K$ is \textit{%
strongly regular }(which means that if a net $(f_{i})_{i\in I}\subseteq
\mathcal{P}$ converges weakly to $f\in \mathcal{P}$ then $Tf_{i}\overset{%
\left\Vert .\right\Vert }{\rightarrow }Tf$ \cite{15}).

A bounded linear operator $T$ from $L^{1}$ to a Banach space $X$ is said to
be \textit{Dunford-Pettis operator} if $T$ maps every weakly compact subset
of $L^{1}$ into a norm compact subset of $X$ (\cite{14}).

\smallskip
{\textit{Indices, trees and bushes.}

In the notation we follow \cite{2}.
If the set of all finite sequences of natural numbers of the form $%
a=(0,a_{1},a_{2},...a_{n})$ is denoted by $%
\mathbb{N}
^{(
\mathbb{N}
)}$, using the notion of length ($\left\vert 0\right\vert =0$, $\left\vert
(0,a_{1},a_{2},...a_{n})\right\vert =n$) and the notion of restriction ( $%
a/n=\left\vert (0,a_{1},a_{2},...a_{n})\right\vert $, if $\left\vert
a\right\vert \geq n$ ) we can define a partial order in $%
\mathbb{N}
^{(
\mathbb{N}
)}$ by $a\leq \beta $ if and only if $\left\vert a\right\vert \leq
\left\vert \beta \right\vert $ and $\beta /\left\vert a\right\vert =a$, when
$a,\beta \in
\mathbb{N}
^{(%
\mathbb{N} )}$. We also make use of the lexicographic total order
of $\mathbb{N}^{(\mathbb{N})}$ and denote it by $<_{\text{lex}}$.
A subset $\mathcal{A}$ of $%
\mathbb{N}
^{(
\mathbb{N}
)}$ is called a \textit{finitely branching tree }if the set $\left\{ a\in
\mathcal{A}:\left\vert a\right\vert =n\right\} $ is finite for every $n\in
\mathbb{N}
$, when $n\leq \left\vert a\right\vert $ and $a/n\in \mathcal{A}$. The set
of the immediate successors of $a\in \mathcal{A}$ is denoted by $%
S_{a}=\left\{ \beta :a<\beta ,\left\vert \beta \right\vert =\left\vert
a\right\vert +1\right\} $ and is finite when $\mathcal{A}$ is a finitely
branching tree.

Let $(\varepsilon_{n})_{n}\subset (0,1)$ be such that 
$\sum\limits_{n=0}^{\infty }\varepsilon _{n}<\frac{\delta }{4}$.
A bounded subset $(x_{a})_{a\in\mathcal{A}}$ of a Banach space $%
X$ is called a $\delta -$\textit{approximate bush} with $\delta >0,$ if and
only if $\mathcal{A}$ is a finitely branching tree, for every $a,\beta \in
\mathcal{A}$ with $\beta \in S_{a}$ we have 
$\lVert x_{a}-x_{\beta
}\rVert >\delta $ and there exists $\{ \lambda _{\beta }:\beta \in
S_{a}\} $ with $\lambda _{\beta }\geq 0$, $\sum\limits_{\beta \in
S_{a}}\lambda _{\beta }=1$ and $\lVert x_{a}-\sum\limits_{\beta \in
S_{a}}\lambda _{\beta }x_{\beta }\rVert <\varepsilon _{\left\vert
a\right\vert }$. The vectors $y_{\beta}=x_{\beta }-x_{a}$, where $\beta\in S_a$, are called the \textit{%
nodes} of the approximate bush.

We have the identity: $\sum\limits_{\left\vert \beta \right\vert
=m}\lambda _{\beta }x_{\beta
}=\sum\limits_{n=0}^{m}\sum\limits_{\left\vert a\right\vert =n}\mu
_{a}y_{a}$ where $\mu _{a}=\lambda _{a}$ for $\left\vert
a\right\vert =m$ and if $m>\left\vert a\right\vert $, $\mu
_{a}=\sum\limits_{\beta \in S_{a}}\mu _{\beta }.$

We can then define the notion of the \textit{average back bush} $(\widetilde{%
x_{a}})_{a\in\mathcal{A}}$ corresponding to the approximate
bush. Set $x_{a}^{m}=\sum\limits_{\left\vert \beta \right\vert =m}\lambda
_{\beta }^{(a)}x_{\beta }$ for $a\in\mathcal{A}$ and $m>\left\vert
a\right\vert $, a convex combination, where the numbers $\lambda _{\beta
}^{(a)}$ are defined inductively. If $m=\left\vert a\right\vert +1$ then $%
x_{a}^{m}=\sum\limits_{\beta \in S_{a}}\lambda _{\beta }x_{\beta }$ with
the numbers $\lambda _{\beta }$ those in the definition of the $\delta -$%
approximate bush and if the numbers $\lambda _{\beta }^{(a)}$ with $%
\left\vert \beta \right\vert =n$ are defined for some $n,$ then set $\lambda
_{\gamma }^{(a)}=\lambda _{\beta }^{(a)}\lambda _{\gamma }$ for $\left\vert
\gamma \right\vert =n+1$ with the numbers $\lambda _{\gamma }$ those in the
definition of the $\delta -$approximate bush. It can be shown that the
sequence $\{x_{a}^{m}\}_{m>\left\vert a\right\vert }$ is norm Cauchy. Define
$\widetilde{x_{a}}=\underset{m\rightarrow \infty }{\lim }x_{a}^{m},$
then for $a\in\mathcal{A}$, $\beta \in S_{a}$ we have $\left\Vert
\widetilde{x_{a}}-\widetilde{x_{\beta }}\right\Vert >\frac{\delta }{2}$ , $%
\widetilde{x_{a}}=\sum\limits_{\beta \in S_{a}}\lambda _{\beta }\widetilde{%
x_{\beta }}$ and every $\widetilde{x_{a}}$ belongs to $\overline{co}%
(x_{a})_{a\in\mathcal{A}}$.

Let $(y_{a})_{a\in\mathcal{A}}$ and $(\widetilde{y_{a}}%
)_{a\in\mathcal{A}}$ be the nodes of the $\delta -$approximate bush $%
(x_{a})_{a\in\mathcal{A}}$ and the nodes of the corresponding
regular bush $(\widetilde{x_{a}})_{a\in\mathcal{A}}$
respectively, when the family $(\mu _{a})_{a\in\mathcal{A}}$ of real
numbers is a normalized conditionally determined family (which means that $%
\mu _{0}=1,$ $\mu _{a}\geq 0,$ and $\mu _{a}=\sum\limits_{\beta \in
S_{a}}\mu _{\beta },\ $\cite{21}).It is true that $\sum\limits_{n=0}^{\infty
}\sum\limits_{\left\vert a\right\vert =n}\mu
_{a}y_{a}=\sum\limits_{n=0}^{\infty }\sum\limits_{\left\vert a\right\vert
=n}\mu _{a}\widetilde{y_{a}}$ whenever either series converges.

\smallskip

\textit{The spaces }$C(\omega ^{\omega ^{k}}).$

Let $\omega $ be the first infinite ordinal number corresponding to $%
\mathbb{N}
$ and $k\in
\mathbb{N}
^{\ast }$.

It is true that $\omega ^{\omega ^{k}}=\sum\limits_{n=0}^{\infty
}\omega ^{\omega ^{k-1}\cdot n}$ and if $C(K)$ is the space of
continuous real functions defined on the set $K$, we have
$C(\omega ^{\omega ^{k}})=\left( \sum\limits_{n=0}^{\infty }\oplus
C(\omega ^{\omega ^{k-1}\cdot n})\right) _{0}$. This can be proved
by the result due to Bessaga and Pelczynski \cite{6}.

\textit{Theorem (Bessaga-Pelczynski):} \textit{If }$a<\beta $\textit{\ are
countable ordinals, then }$C(a)$\textit{\ and }$C(\beta )$\textit{\ are
isomorphic Banach spaces if and only if }$\beta <a^{\omega }$\textit{.}

Also we have that $C(\omega )$ is isomorphic to $c_{0}$ and $C(\omega
^{\omega ^{k}})$ is isomorphic to $C(\omega ^{\omega ^{k}\cdot n})$ for $%
n\in
\mathbb{N}
$.

Finally it is known that $C(\omega ^{\omega })$ and hence $C(a)$ with $%
a>\omega ^{\omega }$ , can not be embedded in a Banach space with
unconditional basis (in fact $C(\omega ^{\omega })$ can not be embedded in a
Banach space with unconditional FDD), (Pelczynski's thesis). See also \cite{1}
(Theorem 4.5.2). Of course that means, that no $C(\omega ^{\omega ^{k}})$
can be embedded in a Banach space with unconditional basis since $C(\omega
^{\omega })$ is a subspace of $C(\omega ^{\omega ^{k}})$ for every $k\in
\mathbb{N}
^{\ast }.$

\smallskip
\textit{The fundamental example.}

In \cite{4} one can find two examples of closed bounded convex subsets of $c_{0}$%
. The first example has the CPCP but fails PCP. The second example has
strong regularity but fails the CPCP.

These examples are the prototype for the following simplified example which
is funtamental for our work.

We denote by $\mathscr{D}$\  the dyadic tree (i.e. the family of all finite
sequences consisting of 0's and 1's), ordered by the initial segment partial
order and we endow $c_{00}(\mathscr{D})$\  with the supremum norm. Clearly
its completion is $c_{0}(\mathscr{D})$. For $a\in \mathscr{D}$, we denote
by $x_{a}=\sum_{\gamma \leqslant a}e_{\gamma }$, where $(e_{a})_{a\in %
\mathscr{D}}$ is the natural basis of $c_{00}(\mathscr{D})$. We also set
\begin{equation*}
\tilde{x}_{a}=x_{a}+\sum_{k=1}^{\infty }\sum_{\substack{ |b|=|a|+k \\ b>a}}%
\frac{1}{2^{k}}e_{b}
\end{equation*}%
Then on the set $K=\cco(\tilde{x}_{a})_{a\in \mathscr{D}}$\  the weak and
norm topologies coincide.
\section{`` Large ''
operators on $L^{1}$ with `` small '' projections}
\begin{proposition} Let $X,$ $X_{n}$ $n\in
\mathbb{N}
$, be Banach spaces. Suppose that $X=\sum\limits_{n=1}^{\infty
}\oplus X_{n}$ and that there exists a non-strongly regular
operator $T:L^{1}(0,1)\rightarrow X$, such that the operators 
$P_{n}T:L^{1}(0,1)\rightarrow X_{n}$ are strongly regular for every 
$n\in
\mathbb{N}
$ where $P_{n}$ denote the projections 
$P_{n}:X\rightarrow X_{n}$.

Then there exists an operator $D:L^{1}(0,1)\rightarrow L^{1}(0,1)$
such that $TD:L^{1}(0,1)\rightarrow X$  is
non-representable and the operators $P_{n}TD:L^{1}(0,1)\rightarrow X_{n}$%
 are representable for every $n\in
\mathbb{N}$.
\end{proposition}
\begin{proof}
Since $T$ is non strongly regular there exists a Borel set $U\subset (0,1)$
and $\delta >0$ such that for every weak open subset $W$ of $\mathcal{P}_{U}$
we have :

\emph{(1) \ \ }$diam(T(W))>2\delta $ (Theorem IV.10, \cite{15})

(where $\mathcal{P}_{U}=\{f\in L^{1}(0,1):f\geq 0,\int f=1,$ $suppf\subset U$%
, are the densities supported in $U$).

Since $P_{n}T$ are strongly regular operators, for every $n\in
\mathbb{N}$ , we get that:

\emph{(2) }$\ \ \ $the maps $P_{n}T:\mathcal{P}_{U}\rightarrow X_{n}$ are
weak to norm continuous.\qquad

Inductively we define $(f_a)_{a\in\mathcal{A}}$ in
$\mathcal{P}_{U}$ satisfying the following properties:
\begin{itemize}
\item[(i)] For every $a\in\mathcal{A}$ and $\beta\in
S_a\quad\|Tf_a - Tf_{\beta}\| > \delta.$ 
\item[(ii)] For all
$a\in\mathcal{A}$ there exists $(\lambda_\beta)_{\beta\in S_a},
\lambda_{\beta}\geq 0, \sum_{\beta\in S_a} \lambda_{\beta} = 1$ and
$\|f_a - \sum_{\beta\in S_a}\lambda_\beta x_\beta\| <
\frac{1}{2^n}$ 
\item[(iii)] For all $n\in\mathbb{N}$, for all
$a\in\mathcal{A}$ with $|a|\geqslant n$ and $\beta\in S_a$,
$\lVert P_nTf_a - P_nTf_\beta\rVert < \frac{1}{2^n}$
\end{itemize}
The construction goes as follows. Assume that $(f_a)_{|a|\leqslant
n}$ has been chosen satisfying the inductive assumptions. Then
setting $\mathcal{A}_n = \{a: |a|=n\}$, for every
$a\in\mathcal{A}_n$ we choose a net $(f_{a,i})_{i\in I_a}\subset
\mathcal{P}_U$ such that $f_{a,i}\stackrel{\text{w}}{\rightarrow}
f_a$ and $\|Tf_{a,i} - Tf_a\|>\delta$.

Since for  $k=1,\ldots,n+1\quad
P_kTf_{a,i}\stackrel{\|\cdot\|}{\rightarrow} P_kTf_a$ we may
assume that the net $(f_{a,i})_{i\in I_a}$ satisfies $\|P_{k}Tf_{a,i}
- P_kTf_a\| < \frac{1}{2^{n+1}}$

By Mazur's theorem there exists a finite subset  $F_a$  of $I_a$
and $(\lambda_i)_{i\in F_a}, \lambda_i\geq 0, \sum_{i\in
F_a}\lambda_i = 1$ such that $\|f_a - \sum_{i\in
F_a}\lambda_if_{a,i}\| < \frac{1}{2^{n+1}}$.

We set $S_a=\{\beta:\beta=(a,i), i\in F_{a}\}$ the finite set of the immediate successors
of $a$ and the family $(f_\beta)_{\beta\in S_a}, |a| = n$ is the
desired one.

Let us point out, that if we don't require the $f_{\beta}$,
$\beta\in S_{a}$ to be different, we may assume
$\lim_{n\to\infty}\max\{\lambda_a:\vert a\vert=n\}=0$.  Let
$\left\{ \xi _{n}\right\} _{n\in \mathbb{N}}$ be the
quasi-martingale which is determined by this bush. Then
$\sigma(\cup_{n\in\mathbb{N}}\sigma(\xi_n))=\mathcal{B}(0,1)$ (
the Borel measurable sets).

For a $\cup_{n\in\mathbb{N}}\sigma(\xi_n)-$simple function
$\varphi$ the  limit $D\varphi=\lim\limits_{n\rightarrow \infty
}\int \xi _{n}(t)\varphi (t)dt$ exists. By density we extend the
operator $D$ on $L^{1}(0,1)$.

Then the operator $TD:L^{1}(0,1)\rightarrow X$ is non-representable, since $%
\forall a,$ $\forall \beta \in S_{a}$ we have $\left\Vert Tf_{\beta
}-Tf_{a}\right\Vert >\delta $, while the operators $P_{n}TD:L^{1}(0,1)%
\rightarrow X_{n}$ are representable for every $n\in
\mathbb{N}
$ since:
For $n\in
\mathbb{N}
$ and $\left\vert \gamma \right\vert =m+k>m=\left\vert a\right\vert $, $%
\gamma >a$ we have

$\left\Vert P_{n}Tf_{\gamma }-P_{n}Tf_{a}\right\Vert \leq \frac{1}{2^{m}}+%
\frac{1}{2^{m+1}}+...+\frac{1}{2^{m+k-1}}<\frac{1}{2^{m-1}}$ and so

$\int \left\Vert P_{n}Tf_{\gamma }-P_{n}Tf_{a}\right\Vert dt<\frac{1}{2^{m-1}%
}$.

Taking $m$ big enough this implies that the bush $(P_{n}Tf_{a})_{a\in
\mathcal{A}}$ is Cauchy in Bochner norm in $X_{n}$, for every $n\in
\mathbb{N}
$ and therefore the operators $P_{n}TD$ are representable \cite{14}.
\end{proof}
Of related interest is the following:
\begin{proposition}
Let $X, X_{n}$,$n\in\mathbb{N}$ be Banach spaces. Suppose that $X=\sum\limits_{n=1}^{\infty
}\oplus X_{n}$  and let $T:L^{1}(0,1)\rightarrow X$ be a
non Dunford-Pettis operator such that the operators 
$P_{n}T:L^{1}(0,1)\rightarrow X_{n}$  are Dunford-Pettis.

Then the conclusion of Proposition 1.1 is true.
\end{proposition}
\begin{proof}
It is shown in \cite{8} that if $T:L^{1}(0,1)\rightarrow X$ is non Dunford-Pettis
operator there exists a dyadic tree $\{\psi _{n,k}: n=0,1,\dots, 1\leq k\leq 2^{n}\}$ in
$L^{1}(0,1)$ so that $(T\psi _{n,k})$ is a $\delta -$%
tree in $X.$ The nodes $d_{n,k}=\psi _{n+1,2k-1}-\psi _{n+1,2k}$ of the tree
$(\psi _{n,k})$ can be taken to be of the form $2\psi
_{n,k}r_{n,k},$ where $r_{n,k}$ are elements from a weakly null sequence $%
(r_{n})_{n\in
\mathbb{N}
}$ in $L^{1}(0,1)$ so that $\underset{n}{\inf }\left\Vert Tr_{n}\right\Vert
>\delta ^{\prime }$ for some $\delta ^{\prime }>0$. Since $P_{i}T$ are
Dunford-Pettis for every $i\in
\mathbb{N}$ we may choose the $\{r_{n,k} :n=0,1,\dots, 1\leq k\leq 2^{n}\}$ in such a way so that
for every $i\in\mathbb{N}
$ there exists a $n_{i}\in
\mathbb{N}$ such that $\sum\limits_{k=1}^{2^{n}}\left\Vert P_{i}Td_{n,k}\right\Vert <%
\frac{1}{2^{n}}$ if $n>n_{i}.$

\noindent Let $D:L^{1}(0,1)\rightarrow L^{1}(0,1)$ be the operator defined by the tree
$(\psi _{n,k})$. It follows that the operators $%
P_{i}TD:L^{1}(0,1)\rightarrow X$ are representable for every $i\in
\mathbb{N}
$ (in fact can be taken to be compact).
\end{proof}
\section{Convex sets on which the norm and the weak topologies coincide}
In this section we show   that under certain conditions there exist closed bounded convex sets
on which the norm topology  coincides with the the weak topology.
\begin{dfn}
Let $(x_a)_{a\in\mathcal{A}}$\; be a $\delta$-approximate bush with $%
(y_a)_{a\in\mathcal{A}}$\; the corresponding nodes. Let also $(\tilde{x}%
_a)_{a\in\mathcal{A}}$\; be the regular averaging back bush resulting from $%
(x_a)_{a\in\mathcal{A}}$. We say that the closed convex set $K = \cco(\tilde{%
x}_a)_{a\in\mathcal{A}}$\; satisfies the \emph{non-atomic martingale
coordinatization property} if every $x\in K$\; is represented as $x =
\sum_{k=0}^\infty\sum_{|a|=k}\lambda_a^{(x)}y_a$, with $\lambda_%
\varnothing^{(x)} = 1, \lambda_a^{(x)}\geqslant 0, \lambda_a^{(x)} =
\sum_{b\in S_a}\lambda_b^{(x)}$\; for all $a\in\mathcal{A}$\; and setting $%
\lambda_k^{(x)} = \max\{\lambda_a^{(x)}: |a| = k\}$, then $\lim_{k\to
\infty} \lambda_k^{(x)} = 0$.
\end{dfn}
\begin{ntt}
In the sequel, for a Banach space $X$ admitting a (not  necessarily finite)
Schauder decomposition  $(X_n)_{n\in\mathbb{N}}$\;(i.e. $X=\sum_{k=1}^\infty%
\oplus X_n$) and  $x\in X$\; we say that $I\subset\mathbb{N}$\; is the
support of  $x$, if $x\in\sum_{n\in I}\oplus X_n$. Also for $%
X=\sum_{k=1}^\infty\oplus  X_n$, a family $(y_a)_{a\in\mathcal{A}}$\; is
said to be block, if  $(y_a)_{a\in\mathcal{A}}$\; have pairwise disjoint
supports with respect  to $(X_n)_{n\in\mathbb{N}}$.
\end{ntt}
\begin{dfn}
Let $X$ be a Banach space with a Schauder decomposition $(X_n)_{n\in\mathbb{N%
}}$. A $\delta$-approximate bush $(x_a)_{a\in\mathcal{A}}$\; is said to be a
\emph{block $\delta$-approximate bush}, if there exists a family $(I_a)_{a\in%
\mathcal{A}}$\; of disjoint intervals of $\mathbb{N}$, such that if $a<_{%
\text{lex}}b$, then $I_a < I_b$, and for every $a\in\mathcal{A},\; \supp%
\{y_a\}\subset I_a$.
\end{dfn}
\begin{lem}
Let $X$ be a Banach space with a Schauder decomposition $(X_n)_{n\in\mathbb{N%
}}$ and a block $\delta$-approximate bush $(x_a)_{a\in\mathcal{A}}$\; in $X$%
. Then for $x\in\cco(\tilde{x}_a)_{a\in\mathcal{A}}$\; there exists a unique
non-atomic martingale coordinatization. \label{lem3}
\end{lem}
\begin{proof}
By definition, each $\tilde{x}_a$\; has a martingale coordinatization for
all $a\in\mathcal{A}$, this evidently then holds for all $x\in\co(\tilde{x}%
_a)_{a\in\mathcal{A}}$.

Let $x\in\cco(\tilde{x}_a)_{a\in\mathcal{A}}, (x_n)_{n\in\mathbb{N}}\subset%
\co(\tilde{x}_a)_{a\in\mathcal{A}}$, such that $x_n\xrightarrow[]{\Vert\cdot\Vert}x$.
Assume that each $x_n = \sum_{k=0}^\infty\sum_{|a|=k}\lambda_a^ny_a$%
\; and $(y_a^*)_{a\in\mathcal{A}}$\; are the biorthogonal functionals of $%
(y_a)_{a\in\mathcal{A}}$, defined on $\overline{<(y_a)_{a\in\mathcal{A}}>}$,
then $y_a^*(x_n)\rightarrow y_a^*(x)$, for all $a\in\mathcal{A}$. Therefore
for each $a\in\mathcal{A}$, there exists $\lambda_a^{(x)}$\; such that $%
\lambda_a^n\rightarrow \lambda_a^{(x)}$. Then $x=\sum_{k=0}^\infty%
\sum_{|a|=k}\lambda_a^{(x)}y_a$\; and this coordinatization is unique.

Also, it is non atomic, since if $x =
\sum_{k=0}^\infty\sum_{|a|=k}\lambda_a^{(x)}y_a$\; and $\varepsilon>0$, then
there exist $n_0\in\mathbb{N}$, such that $\|\sum_{k=n}^\infty\sum_{|a|=k}%
\lambda_a^{(x)}y_a\| < \varepsilon$, for all $n\geqslant n_0$, thus if $|a|
= n \geqslant n_0$, then
\begin{equation*}
\lambda_a^{(x)}\|y_a\| \leqslant C\Big\|\sum_{k=n}^\infty\sum_{|a|=k}%
\lambda_a^{(x)}y_a\Big\| \leqslant C\varepsilon
\end{equation*}
This yields that $\lambda_a^{(x)} \leqslant \frac{C\varepsilon}{\delta}$\;
and hence $\lambda_k^{(x)} = \max\{\lambda_a^{(x)}: |a| = k\}\rightarrow 0$.
\end{proof}
\begin{lem}
Let $X$ be a Banach space with a Schauder decomposition $(X_n)_{n\in\mathbb{N%
}}$\; and $(x_a)_{a\in\mathcal{A}}$\; be a $\delta$-approximate bush in $X$.
Assume moreover that there exists a block $\delta^\prime$-approximate bush $%
(x_a^\prime)_{a\in\mathcal{A}}$, such that by setting $(y_a)_{a\in\mathcal{A}%
}, (y_a^\prime)_{a\in\mathcal{A}}$\; the corresponding families of nodes, we
have that $\|y_a - y_a^\prime\| < \delta_a$\; and\; $\sum_{a\in\mathcal{A}%
}\delta_a < \infty$.

Then the set $K = \cco(\tilde{x}_a)_{a\in\mathcal{A}}$\; satisfies the
non-atomic martingale coordinatization property. \label{lem4}
\end{lem}
\begin{proof}
Let $x\in K, (x_n)_{n\in\mathbb{N}}\subset\co(\tilde{x}_a)_{a\in\mathcal{A}}$%
\; with $x_n\overset{\|\cdot\|}{\longrightarrow}x$. If $x_n =
\sum\limits_{k=0}^\infty\sum\limits_{|a|=k}\lambda_a^ny_a$, since the set $\mathcal{A}$\;
is countable and $(\lambda_a^n)_{n\in\mathbb{N}}$\; is bounded for all $a\in%
\mathcal{A}$, by passing to a subsequence we may assume that $%
\lambda_a^n\rightarrow\lambda_a^{(x)}$, for all $a\in\mathcal{A}$.

Define $x_n^\prime = \sum_{k=0}^\infty\sum_{|a|=k}\lambda_a^ny_a^\prime$. By
the fact that $\sum_{a\in\mathcal{A}}\|y_a-y_a^\prime\|<\infty, x_n^\prime$%
\; is well defined and $x_n^\prime\in\cco(\tilde{x}_a^\prime)_{a\in\mathcal{A%
}}$. It will be shown that $(x_n^\prime)_{n\in\mathbb{N}}$\; is a Cauchy
sequence.

Let $\varepsilon>0$. There exists $n_0\in\mathbb{N}$, such that for all $%
n,m\geqslant n_0,\;\|x_n-x_m\| <\frac{\varepsilon}{3}$. There also exists $%
\ell_0\in\mathbb{N}$, such that $\sum_{|a|>\ell_0}\|y_a-y_a^\prime\|<\frac{%
\varepsilon}{6}$. Moreover, there exists $n_1\geqslant n_0$, such that for
all $n,m\geqslant n_1$, for all $a\in\mathcal{A}$\; with $|a|\leqslant
\ell_0:\; |\lambda_a^n - \lambda_a^m| < \frac{\varepsilon}{3M}$, where $M =
\sum_{a\in\mathcal{A}}\|y_a-y_a^\prime\|$.

Then, for $n,m\geqslant n_1$:
\begin{eqnarray*}
\|x_n^\prime - x_m^\prime\| &=& \big\|\sum_{k=0}^\infty\sum_{|a|=k}(%
\lambda_a^n - \lambda_a^m)y_a^\prime\big\| \\
&=& \big\|\sum_{k=0}^\infty\sum_{|a|=k}(\lambda_a^n -
\lambda_a^m)(y_a^\prime-y_a) + x_n - x_m\big\| \\
&\leqslant& \big\|\sum_{k=0}^{\ell_0}\sum_{|a|=k}(\lambda_a^n -
\lambda_a^m)(y_a^\prime-y_a)\big\| \\
&& +\; \big\|\sum_{k=\ell_0 +1}^{\infty}\sum_{|a|=k}(\lambda_a^n -
\lambda_a^m)(y_a^\prime-y_a)\big\| + \|x_n - x_m\| \\
&\leqslant &\max\{|\lambda_a^n - \lambda_a^m| : |a|\leqslant
\ell_0\}\sum_{k=0}^{\ell_0}\sum_{|a|=k}\|y_a^\prime-y_a\| \\
&& +\; \sup\{|\lambda_a^n - \lambda_a^m| : |a|>
\ell_0\}\sum_{k=\ell_0+1}^{\infty}\sum_{|a|=k}\|y_a^\prime-y_a\| \\
&& +\;\|x_n - x_m\| \\
&\leqslant& \frac{\varepsilon}{3M}M + 2\frac{\varepsilon}{6} + \frac{%
\varepsilon}{3} = \varepsilon
\end{eqnarray*}
Hence $(x_n^\prime)_{n\in\mathbb{N}}$\; is converging to some $x^\prime\in%
\cco(\tilde{x}_a^\prime)_{a\in\mathcal{A}}$. As in the previous proof, if we
consider $(y_a^{\prime*})_{a\in\mathcal{A}}$\; the biorthogonal functionals
of $(y_a^\prime)_{a\in\mathcal{A}}$ defined on the space $\overline{%
<(y_a^\prime)_{a\in\mathcal{A}}>}$, then $y_a^{\prime*}(x_n^\prime)%
\rightarrow y_a^{\prime*}(x^\prime)$\; for all $a\in\mathcal{A}$. Thus $%
x^\prime = \sum_{k=0}^\infty\sum_{|a|=k}\lambda_a^{(x)}y_a^\prime$\; and by
virtue of Lemma \ref{lem3}, $(\lambda_a^{(x)})_{a\in\mathcal{A}}$\; is a
non-atomic martingale coordinatization.

As before, $y = \sum_{k=0}^\infty\sum_{|a|=k}\lambda_a^{(x)}y_a$\; is well
defined and $y\in K$. It remains to be shown that $y=x$.

Towards a contradiction, suppose that $x_n\nrightarrow y$. By passing to an
appropriate subsequence, there exists $\varepsilon>0$, such that $\|x_n -
y\| > \varepsilon$\; for all $n\in\mathbb{N}$. There also exists $\ell_0\in%
\mathbb{N}$\; such that $\sum_{|a|\geqslant\ell_0}\|y_a-y_a^\prime\| < \frac{%
\varepsilon}{10}$, moreover there exists $n_1\in\mathbb{N}$\; such that for
all $n\geqslant n_1,\;\|\sum_{k\leqslant\ell_0}\sum_{|a| = k}(\lambda_a^n -
\lambda_a^{(x)})y_a\| < \frac{\varepsilon}{10}$. Hence we have for $%
n\geqslant n_1$:
\begin{eqnarray*}
\big\|\sum_{k>\ell_0}\sum_{|a| = k}(\lambda_a^n - \lambda_a^{(x)})y_a\big\| %
&=& \big\|x_n - y - \sum_{k\leqslant\ell_0}\sum_{|a| = k}(\lambda_a^n -
\lambda_a^{(x)})y_a\big\| \\
&\geqslant & \|x_n - y\| - \big\|\sum_{k\leqslant\ell_0}\sum_{|a| =
k}(\lambda_a^n - \lambda_a^{(x)})y_a\big\| \\
&>& \varepsilon - \frac{\varepsilon}{10} = \frac{9\varepsilon}{10}
\end{eqnarray*}
Also,
\begin{eqnarray*}
\big\|\sum_{k>\ell_0}\sum_{|a| = k}(\lambda_a^n - \lambda_a^{(x)})y_a\big\| %
&=& \big\|\sum_{k>\ell_0}\sum_{|a| = k}(\lambda_a^n - \lambda_a^{(x)})(y_a -
y_a^\prime) \\
&& + \sum_{k>\ell_0}\sum_{|a| = k}(\lambda_a^n - \lambda_a^{(x)})y_a^\prime%
\big\| \\
&\leqslant & \big\|\sum_{k>\ell_0}\sum_{|a| = k}(\lambda_a^n -
\lambda_a^{(x)})(y_a - y_a^\prime)\big\| \\
&&+\; \big\|P_{|a|>\ell_0}(x_n^\prime - x^\prime)\big\| \\
&\leqslant & 2\frac{\varepsilon}{10} + 2C^{}\|x_n^\prime - x^\prime\|
\end{eqnarray*}
Here $P_{|a|>\ell_0}(x)=x-\sum_{\vert a\vert\leq\ell_{0}}\tilde{P}_{\alpha}(x)$, 
where $\tilde{P}_{a}(x)=\sum_{i\in I_{a}}P_{i}(x)$, $P_{i}:X\to X_{i}$ are the natural projections of the decomposition
and $C$ the constant of the decomposition.

By
choosing $n$\; sufficiently large, we have  $\|\sum_{k>\ell_0}\sum_{|a|
= k}(\lambda_a^n - \lambda_a^{(x)})y_a\| < \frac{9\varepsilon}{10}$, a
contradiction that concludes our proof.
\end{proof}
\begin{lem}
Let $X$\; be a Banach space, $\mathcal{A}$\; a finitely branching tree, $%
(y_a)_{a\in\mathcal{A}}$, $(y_a^\prime)_{a\in\mathcal{A}}$\; subsets of $X,
(\varepsilon_n)_{n=0}^\infty$\; a sequence of positive reals with $%
\sum_{n=0}^\infty\varepsilon_n < \infty$\; and $\|y_a -
y_a^\prime\|<\varepsilon_{|a|}$\; for all $a\in\mathcal{A}$. Define:
\begin{equation*}
K=\Big\{x\in X: x =
\sum_{k=0}^\infty\sum_{|a|=k}\lambda_a^{(x)}y_a,\,\lambda_\varnothing^{(x)}
=1, \lambda_a^{(x)}\geqslant 0, \lambda_a^{(x)} = \sum_{b\in
S_a}\lambda_b^{(x)},\,a\in\mathcal{A} \Big\}
\end{equation*}
Suppose $L$\; is a subset of $K$\; and that on the set
\begin{equation*}
L^\prime = \Big\{x\in X: x =
\sum_{k=0}^\infty\sum_{|a|=k}\lambda_a^{(x)}y_a^\prime,\;\text{with}%
\;\sum_{k=0}^\infty\sum_{|a|=k}\lambda_a^{(x)}y_a\in L \Big\}
\end{equation*}
the weak and norm topologies coincide. Then on $L$\; the weak and norm
topologies also coincide. \label{lem5}
\end{lem}
\begin{proof}
Define $(r_a)_{a\in\mathcal{A}}$\; with $r_a = \sum_{\gamma\leqslant
a}(y_\gamma - y_\gamma^\prime)$. It will be shown that the set $(r_a)_{a\in%
\mathcal{A}}$\; is totally bounded.

Let $\varepsilon>0$. There exists $n_0\in\mathbb{N}$, such that $%
\sum_{n\geqslant n_0}\varepsilon_n < \varepsilon$. Let $\gamma\in\mathcal{A}%
, |\gamma|\geqslant n_0$. Then there exists $a\in\mathcal{A},\; |a| = n_0,\;
a\leqslant\gamma$. We have
\begin{equation*}
\|r_\gamma - r_a\| = \big\|\sum_{a<\delta\leqslant\gamma}(r_\delta -
r_{\delta^-})\big\| = \big\|\sum_{a<\delta\leqslant\gamma}(y_\delta -
y_\delta^\prime)\big\| \leqslant \sum_{a<\delta\leqslant\gamma}\|y_\delta -
y_\delta^\prime\| < \varepsilon
\end{equation*}
Thus the set $(r_a)_{a\in\mathcal{A}}$\; is totally bounded and this means
that $\cco(r_a)_{a\in\mathcal{A}}$\; is norm compact.

Let $x\in L,\; x=\sum_{k=0}^\infty\sum_{|a|=k}\lambda_a^{(x)}y_a$. Since $%
\|y_a - y_a^\prime\| < \varepsilon_{|a|}$, we conclude that $\sum_{k=0}^\infty%
\sum_{|a|=k}\lambda_a^{(x)}y_a^\prime\in L^\prime$\; and $x =
\sum_{k=0}^\infty\sum_{|a|=k}\lambda_a^{(x)}(y_a - y_a^\prime) +
\sum_{k=0}^\infty\sum_{|a|=k}\lambda_a^{(x)}y_a^\prime$. Then we have that
\begin{eqnarray*}
\sum_{k=0}^\infty\sum_{|a|=k}\lambda_a^{(x)}(y_a - y_a^\prime) &=&
\lim_{n\to\infty}\sum_{k=0}^n\sum_{|a|=k}\lambda_a^{(x)}(y_a - y_a^\prime) \\
&=& \lim_{n\to\infty}\sum_{|a|=n}\lambda_a^{(x)}\left(\sum_{\gamma\leqslant
a}(y_\gamma - y_\gamma^\prime)\right) \\
&=& \lim_{n\to\infty}\sum_{|a|=n}\lambda_a^{(x)}r_a\;\in\cco(r_a)_{a\in%
\mathcal{A}}
\end{eqnarray*}
This means that $L\subset \cco(r_a)_{a\in\mathcal{A}} + L^\prime$. Since $%
\cco(r_a)_{a\in\mathcal{A}}$ is norm compact and on $L^\prime$\; the weak
and norm topologies coincide, it can easily be seen that on $\cco(r_a)_{a\in%
\mathcal{A}} + L^\prime$\; the weak and norm topologies coincide, this of
course means that the same is true for $L$.
\end{proof}
\begin{dfn}
Let $X$ be a Banach space with a Schauder decomposition $(X_n)_{n\in\mathbb{N%
}},\newline
\mathcal{A}$\; a finitely branching tree and $(y_a)_{a\in\mathcal{A}}$\; a
subset of X. Then $(y_a)_{a\in\mathcal{A}}$\; is called \emph{eventually
block}, if there exists $n_0\in\mathbb{N},\; (I_a)_{|a|\geqslant n_0}$\; a
family of disjoint intervals of $\mathbb{N}$, such that if $a<_{\text{lex}}b$%
, then $I_a < I_b$, and for every $a\in\mathcal{A},\; \supp\{y_a\}\subset I_a
$.
\end{dfn}
\begin{rem}
For some $a\in\mathcal{A},\;|a|\geqslant n_0$\; it may occur that $y_a = 0$.
\end{rem}
\begin{lem}
Let $X, X_k, k\in\mathbb{N}$\; be Banach spaces with $X =
\left(\sum_{k=1}^\infty\oplus X_k\right)_0 = \{(x_k)_{k\in\mathbb{N}}:
x_k\in X_k,\;\text{and}\; \lim_{k\to\infty}\|x_k\| = 0\},\; (y_a)_{a\in%
\mathcal{A}}$\; bounded and eventually block. Consider the set
\begin{eqnarray*}
L &=& \Big\{x\in X: x = \sum_{k=0}^\infty\sum_{|a|=k}\lambda_a^{(x)}y_a\;%
\text{with}\;\lambda_\varnothing^{(x)} = 1,\;\lambda_a^{(x)}\geqslant
0,\;\lambda_a^{(x)} = \sum_{b\in S_a}\lambda_b^{(x)} \\
&&\text{for all}\;a\in\mathcal{A}\;\text{and}\;\lim_{k\to\infty}\max\{%
\lambda_a^{(x)}: |a| = k\} = 0\Big\}
\end{eqnarray*}
Then on $L$\; the weak and norm topologies coincide. \label{lem7}
\end{lem}
\begin{proof}
We shall first prove the lemma with the additional assumption that each $%
y_a\neq 0$. Let $x\in L,\; x =
\sum_{k=0}^\infty\sum_{|a|=k}\lambda_a^{(x)}y_a,\;\varepsilon > 0$. It will
be shown that there exists $U$, a relative weak neighbourhood of $x$\; in $L$%
, such that $\diam\{U\} < \varepsilon$, hence $x$ will be a point of
continuity.

Since $<\{y_a : |a|<n_0\}>$\; is finite dimensional, there exists $%
n_1\geqslant n_0$, such that $<\{y_a : |a|<n_0\}>\bigcap\overline{<\{y_a :
|a|\geqslant n_1\}>} = \{0\}$.

Indeed, if $\{x_1,\ldots,x_j\}$\; is a Hamel basis of
\begin{equation*}
<\{y_a : |a|<n_0\}>\bigcap\overline{<\{y_a : |a|\geqslant n_0\}>}
\end{equation*}
then $x_i = \sum_{k=n_0}^\infty\sum_{|a|=k}\mu_a^iy_a$, for $i = 1,\ldots,j$%
. Pick $a_1,\ldots,a_j\in\mathcal{A}$\; with $|a_i|\geqslant n_0$\; and $%
\mu_{a_i}^i\neq 0$\; for $i=1,\ldots,j$. Set $n_1 = \max\{|a_i| : i =
1,\ldots,j\} +1$.

If $M = \sup\{\|y_a\| : a\in\mathcal{A}\}$, there exists $n_2\geqslant n_1$,
such that $\max\{\lambda_a^{(x)}: |a| = k\} < \frac{\varepsilon}{16M}$, for
all $k\geqslant n_2$. Then $x = x_1 + x_2$, where $x_1 = \sum\limits_{k=0}^{n_2 -
1}\sum\limits_{|a|=k}\lambda_a^{(x)}y_a$, $x_2 =
\sum_{k=n_2}^\infty\sum_{|a|=k}\lambda_a^{(x)}y_a$. Define $\ell = \#\{a\in%
\mathcal{A}: |a|\leqslant n_2\}$, $\varepsilon^\prime = \frac{\varepsilon}{%
8(\ell^2M + 2M)}$.

Consider the biorthogonal functionals $(y_a^*)_{|a| = n_2}$\; defined on the
space\newline
$\overline{<\{y_a : a\in\mathcal{A}\}>}$\; with
\begin{equation*}
y_a^*(y_\gamma) = \left\{
\begin{array}{rl}
1 & \text{if } a = \gamma \\
&  \\
0 & \text{otherwise}%
\end{array}
\right.
\end{equation*}
This is possible by the fact that $n_2\geqslant n_1$ and the assumption that
$y_a\neq 0$\; for all $a\in\mathcal{A}$.

Define $U = \big\{y\in L: |y_a^*(y - x)| < \varepsilon^\prime,\; |a| = n_2%
\big\}$\; and let $y\in U$, such that $y =
\sum_{k=0}^\infty\sum_{|a|=k}\lambda_a^{(y)}y_a$. Then $y = y_1 +
y_2$\; where $y_1 = \sum_{k=0}^{n_2 -
1}\sum_{|a|=k}\lambda_a^{(y)}y_a$\; and $y_2 =
\sum_{k=n_2}^\infty\sum_{|a|=k}\lambda_a^{(y)}y_a$.

For $a\in\mathcal{A},\; |a|=n_2$,\quad$|\lambda_a^{(y)} - \lambda_a^{(x)}| =
|y_a^*(y - x)| < \varepsilon^\prime\quad\Rightarrow\quad$ $\lambda_a^{(y)} <
\varepsilon^\prime + \frac{\varepsilon}{16M}$, for all $a\in\mathcal{A},
|a|\geqslant n_2$.

For $a\in\mathcal{A},\; |a|<n_2$,\quad$|\lambda_a^{(y)} - \lambda_a^{(x)}| = %
\big|\displaystyle{\sum_{\substack{ |b|=n_2 \\ a<b}}}(\lambda_b^{(y)} -
\lambda_b^{(x)})\big| < \varepsilon^\prime\ell\quad\Rightarrow\quad$ $\|y_1
- x_1\| = \big\|\sum_{k=0}^{n_2 - 1}\sum_{|a|=k}(\lambda_a^{(y)} -
\lambda_a^{(x)})y_a\big\| \leqslant \varepsilon^\prime\ell^2M$.

Also we have
\begin{align*}
\|y_2 - x_2\| &= \Big\|\sum_{k=n_2}^\infty\sum_{|a|=k}(\lambda_a^{(y)} -
\lambda_a^{(x)})y_a\Big\| 
= \sup\{\|(\lambda_a^{(y)} - \lambda_a^{(x)})y_a\|
: |a|\geqslant n_2\} \\
&\leq  \sup\{(|\lambda_a^{(y)}| + |\lambda_a^{(x)}|)M: |a|\geqslant
n_2\} 
\\
&\leq \left(\frac{\varepsilon}{16M} + \varepsilon^\prime + \frac{%
\varepsilon}{16M}\right)M= \left(\varepsilon^\prime + \frac{\varepsilon}{8M}\right)M
\end{align*}
Then
\begin{eqnarray*}
\|y - x\| &\leqslant & \|y_1 - x_1\| + \|y_2 - x_2\| \leqslant
\varepsilon^\prime\ell^2M + \varepsilon^\prime M + \frac{\varepsilon}{8} \\
&=&\frac{\varepsilon}{8(\ell^2M + 2M)}(\ell^2M + 2M) + \frac{\varepsilon}{8}
= \frac{\varepsilon}{4}
\end{eqnarray*}
Thus $\diam\{U\} \leqslant \frac{\varepsilon}{2} < \varepsilon$.

This completes the proof for the case that each $y_a\neq 0$. For the general
case, we reduce the proof to the previous one as follows. Choose $%
(\varepsilon_n)_{n=0}^\infty$\; a sequence of positive reals with $%
\sum_{n=0}^\infty\varepsilon_n < \infty$\; and define $(y_a^\prime)_{a\in%
\mathcal{A}}$\; with the rule:
\begin{equation*}
y_a^\prime = \left\{
\begin{array}{rl}
y_a & \text{if } y_a \neq 0 \\
&  \\
y_a^{\prime\prime} & \text{otherwise}%
\end{array}
\right.
\end{equation*}
where $\supp\{y_a^{\prime\prime}\}\subset I_a$\; and $0<\|y_a^{\prime\prime}%
\|\leqslant \varepsilon_{|a|}$. We observe that $\|y_a - y_a^\prime\|
\leqslant \varepsilon_{|a|}$\; and by the previous case and Lemma \ref{lem5}%
\; the result follows.
\end{proof}
\section{The main Theorems}
This section contains the main results of the paper. Among other things we
show that the KMP is equivalent wth the RNP on the subsets of $C(\omega
^{\omega ^{k}}).$ In fact we show something stronger, namely, every
non-dentable subset of $C(\omega ^{\omega ^{k}})$ contains a convex closed
subset $L$ such that $L$\ has PCP and fails RNP.
\begin{prop}
Let $Y, Y_k, X, Z, Z_n, Z_{n,k},\;n,k\in\mathbb{N}$\; be Banach spaces such
that $Y = \sum_{k=1}^\infty\oplus Y_k,\; Z_n = \big(\sum_{k=1}^\infty\oplus
Z_{n,k}\big)_0,\; X\hookrightarrow Y$\; and $X$\; contains no copy of $%
\ell_1(\mathbb{N})$. Let $Q_n:X\rightarrow Z_n,\;n\in\mathbb{N}$\; be
bounded linear operators, $K$\; a closed, convex, bounded, non-PCP subset of
$X$\; and suppose that on $P_k(K), R_{n,k}Q_n(K)$\; the weak and norm
topologies coincide for all $n,k\in\mathbb{N}$\;(where $P_k:Y\rightarrow Y_k$%
,\;$R_{n,k}:Z_n\rightarrow Z_{n,k}$\; denote the projections).

Then there exists $L$\; closed, convex, non-dentable subset of K, such that
on $Q_n(L)$\; the weak and norm topologies coincide for all $n\in\mathbb{N}$%
. \label{prop8}
\end{prop}
\begin{proof}
Since $K$ is non-PCP there exists $W\;2\delta$-non-PCP subset of $K$, for
some $\delta>0$. We will inductively construct:
\begin{itemize}
\item[(a)] a $\delta$-bush $(x_a)_{a\in\mathcal{A}}\subset W$
\item[(b)] a family $(I_a)_{a\in\mathcal{A}}$ of disjoint intervals of $%
\mathbb{N}$, such that if $a<_{\text{lex}}b$, then $I_a < I_b$
\end{itemize}
Such that:
\begin{enumerate}
\item[(1)] if $x_a^\prime = \sum_{\gamma\leqslant a}P_{I_\gamma}(y_\gamma)$,
then $(x_a^\prime)_{a\in\mathcal{A}}$\; is a block $\frac{\delta}{2}$%
-approximate bush in $Y$\; and $\sum_{a\in\mathcal{A}}\|y_a - y_a^\prime\| <
\infty$.
\item[(2)] if $R_{n,I_a}Q_n(y_a) = y_a^n$, then $(y_a^n)_{a\in\mathcal{A}}$
is eventually block and 
\begin{center}$\sum_{a\in\mathcal{A}}\|Q_n(y_a) - y_a^n\| < \infty$%
\; for all $n\in\mathbb{N}$.
\end{center}
\end{enumerate}
By taking this construction for granted, it will now be shown that by
setting $L=\cco(\tilde{x}_a)_{a\in\mathcal{A}}$, the desired result is
achieved.

By (1) and Lemma \ref{lem4}, $L$\; satisfies the non-atomic martingale
coordinatization property. Consider the set
\begin{align*}
L_n^{\prime}=\Big\{x\in Z_n&: x = \sum_{k=0}^\infty\sum_{|a|=
k}\lambda_a^{(x)}y_a^n,\;\lambda_\varnothing^{(x)} = 1,\;
\lambda_a^{(x)}\geqslant 0,\; \lambda_a^{(x)} = \sum_{b\in
S_a}\lambda_b^{(x)}, \\
&\text{for all}\;a\in\mathcal{A},\;\lim_{k\to\infty}\max\{\lambda_a^{(x)} :
|a| = k\} = 0\Big\}
\end{align*}
Then by (2) and Lemma \ref{lem4}, on $L_n^\prime$\; the weak and norm
topologies coincide. Now define
\begin{align*}
L_n=\Big\{x\in Z_n&:x=\sum_{k=0}^\infty\sum_{|a|=
k}\lambda_a^{(x)}Q_n(y_a),\,\lambda_\varnothing^{(x)} = 1,\,
\lambda_a^{(x)}\geqslant 0,\, \lambda_a^{(x)} = \sum_{b\in
S_a}\lambda_b^{(x)}, \\
&\text{for all}\;a\in\mathcal{A},\;\lim_{k\to\infty}\max\{\lambda_a^{(x)} :
|a| = k\} = 0\Big\}
\end{align*}
By (2) and Lemma \ref{lem5}, on $L_n$\; the weak and norm topologies
coincide. But $L$ has the non-atomic martingale coordinatization property,
thus $Q_n(L)\subset L_n$, hence on $Q_n(L)$\; the weak and norm topologies
coincide, for all $n\in\mathbb{N}$.

In order to complete the proof, we shall now proceed to the previously
mentioned construction.

An important ingredient is the following fact:

If $X$ contains no copy of $\ell_{1}$, $K$  a bounded subset of $X$ and $x\in\overline{K}^{w}$, then there exists a sequence $(x_n)_{n\in\mathbb{N}}$, such that $w-\lim_{n\to\infty}x_n=x$ (see \cite{11},\cite{20}).

Choose $x_\varnothing = x \in W$. Since $X$\; contains no copy of $\ell_1(%
\mathbb{N})$\; and $W$\; is $2\delta$-non-PCP, there exists a sequence $%
(x_m)_{m\in\mathbb{N}}\subset W$, such that $x_m\overset{w}{\rightarrow}x$\;
and $\|x_m - x\| > \delta$, for all $m\in\mathbb{N}$.

For $\varepsilon_0>0$, there exists $k_0\in\mathbb{N}$, such that:
\begin{equation*}
\big\|P_{[1,k_0]}(x) - x\big\| < \varepsilon_0.\qquad\text{Define}\quad
I_\varnothing = [1,k_0]
\end{equation*}
For $\varepsilon_1>0$, since $\big\|P_k(x_m - x)\big\|,\big\|%
R_{1,k}Q_1(x_m-x)\big\|\xrightarrow{m\to\infty}0$\; for all $k\in\mathbb{N}$%
, there exists $m_1\in\mathbb{N}$\; such that $\big\|P_{[1,k_0]}(x_{m_1} - x)%
\big\|,\big\|R_{1,[1,k_0]}Q_1(x_{m_1}-x)\big\| < \frac{\varepsilon_1}{2\cdot
2}$.

There exist $k_1 > k_0$\; such that 
\begin{center}
$\big\|P_{[1,k_1]}(x_{m_1} - x) -
(x_{m_1} - x)\big\|,\big\|R_{1,[1,k_1]}Q_1(x_{m_1}-x) - Q_1(x_{m_1}-x)\big\| %
< \frac{\varepsilon_1}{2\cdot 2}$.
\end{center} Then
\begin{equation*}
\big\|P_{(k_0,k_1]}(x_{m_1} - x) - (x_{m_1} - x)\big\|,\big\|%
R_{1,(k_0,k_1]}Q_1(x_{m_1}-x) - Q_1(x_{m_1}-x)\big\| < \frac{\varepsilon_1}{2%
}
\end{equation*}
Define $I_1^1 = (k_0, k_1]$.

Inductively choose $(x_{m_i})_{i\in\mathbb{N}},\;(I_i^1)_{i\in\mathbb{N}}$\;
a subsequence of $(x_m)_{m\in\mathbb{N}}$\; and successive intervals of $%
\mathbb{N}$, such that:
\begin{equation*}
\big\|P_{I_i^1}(x_{m_i} - x) - (x_{m_i} - x)\big\|,\big\|%
R_{1,I_i^1}Q_1(x_{m_i}-x) - Q_1(x_{m_i}-x)\big\| < \frac{\varepsilon_1}{2^i}
\end{equation*}
For $\delta_0>0$, by Mazur's theorem, there exists a finite set $(x_b)_{b\in
S_\varnothing}\subset (x_{m_i})_{i\in\mathbb{N}}$\; and positive reals $%
(\lambda_b)_{b\in S_\varnothing}$\; with $\sum_{b\in S_\varnothing}\lambda_b
= 1$, such that 
\begin{center}
$\big\|x_\varnothing - \sum_{b\in S_\varnothing}\lambda_b x_b%
\big\| < \delta_0$.
\end{center} Define $(I_b)_{b\in S_a}$\; the corresponding intervals.
Then we have
\begin{equation*}
\sum_{b\in S_\varnothing}\big\|P_{I_b}(y_b) - y_b\big\|,\sum_{b\in
S_\varnothing}\big\|R_{1,I_b}Q_1(y_b) - Q_1(y_b)\big\| < \sum_{i=1}^\infty%
\frac{\varepsilon_1}{2^i} = \varepsilon_1
\end{equation*}
Suppose that $(x_a)_{|a|\leqslant j},\; (I_a)_{|a|\leqslant j}$\; have been
chosen such that, if $|a|,|b|\leqslant j,\; a<_{\text{lex}}b$, then $I_a <
I_b$, $\big\|x_a - \sum_{b\in S_a}\lambda_b x_b\big\| < \delta_{|a|}$,\; $%
\|x_a - x_b\| > \delta$, for $|a| < j$,\;$b\in S_a$, and also
\begin{equation*}
\sum_{|a|=i}\big\|P_{I_a}(y_a) - y_a\big\|,\sum_{|a|=i}\big\|%
R_{\ell,I_a}Q_\ell(y_a) - Q_\ell(y_a)\big\| < \varepsilon_i,\quad\text{for}%
\;1 \leqslant\ell\leqslant i\leqslant j
\end{equation*}
Enumerate the set $\{a: |a|=j\}$\; in lexicographic order and for $a_1$, if $%
N = \#\{a: |a|=j\}$, for $\varepsilon_{j+1}, \delta_j$, as before choose $%
(x_b)_{b\in S_{a_1}}$,\;$(I_b)_{b\in S_{a_1}}$, such that $%
(I_b)_{|b|\leqslant j} < (I_b)_{b\in S_{a_1}}$,\;$\|x_{a_1} - x_b\|>\delta$%
,\;$\|x_{a_1} - \sum_{b\in S_{a_1}}\lambda_b x_b\|<\delta_j$\; and
\begin{equation*}
\sum_{b\in S_{a_1}}\big\|P_{I_b}(y_b) - y_b\big\|,\sum_{b\in S_{a_1}}\big\|%
R_{\ell,I_b}Q_\ell(y_b) - Q_\ell(y_b)\big\| < \frac{\varepsilon_{j+1}}{N}%
,\quad\text{for}\;1\leqslant\ell\leqslant j+1
\end{equation*}
Continue in the same manner for the rest of the set $\{a: |a|=j\}$. Then we
have
\begin{equation*}
\sum_{|a|=j+1}\big\|P_{I_a}(y_a) - y_a\big\|,\sum_{|a|=j+1}\big\|%
R_{\ell,I_a}Q_\ell(y_a) - Q_\ell(y_a)\big\| < \varepsilon_{j+1},\quad\text{%
for}\;1\leqslant\ell\leqslant j+1
\end{equation*}

The inductive construction is complete. If the sequences $%
(\varepsilon_j)_{j=0}^\infty$,\;$(\delta_j)_{j=0}^\infty$\; have been
suitably chosen, then the conclusion of the theorem holds. In fact they need
to be chosen in such a way that $\sum_{j=0}^\infty\varepsilon_j,\sum_{j=0}^%
\infty\delta_j < \frac{\delta}{16}$.

Then it easy to see that:
\begin{itemize}
\item[(i)] $(x_a)_{a\in\mathcal{A}}$\; is a $\delta$-approximate bush

\item[(ii)] $(x_a^\prime)_{a\in\mathcal{A}}$\; is a block $\frac{\delta}{2}$%
-approximate bush and $\sum_{a\in\mathcal{A}}\|y_a - y_a^\prime\|<\infty$.
\item[(iii)] $(y_a^n)_{|a|\geqslant n}$\; is block and $\sum_{a\in\mathcal{A}%
}\big\|y_a^n - Q_n(y_a)\big\|<\infty$, for all $n\in\mathbb{N}$.
\end{itemize}
\end{proof}
\begin{rem}
The proof of Proposition 3.1 shows that if $X$
contains no copy of $l_{1}$ and $X$ fails the PCP then there exists a $%
\delta -$approximate bush $(x_{a})_{a\in \mathcal{A}}$ whose nodes form a
basic sequence. Therefore $X$ contais a subspace with a basis that fails the
RNP. This result is known to the experts but we were unable to trace a
reference. In \cite{3} a Banach space $X$ is contructed so that $X^{\ast }$ is
separable and the PCP is equivalent with the RNP on the subsets of $X.$ It
follows that if a subspace $Y$ of $X$ fails RNP then $Y$ contains a space $Z$
with a basis that fails the RNP.
\end{rem}
\begin{thm}
Let $X$ be a separable Banach space that contains no copy of $\ell_1(\mathbb{%
N})$\; and $Q_n:X\rightarrow C(\omega^{\omega^k}),\;n\in\mathbb{N}$\; be
bounded linear operators. Suppose $K$\; is a closed, convex, bounded,
non-PCP subset of $X$, such that the PCP and RNP are equivalent on the
subsets of $K$. Then there exists $L$\; closed, convex, bounded,
non-dentable subset of $K$, such that on $Q_n(L)$\; the weak and norm
topologies coincide for all $n\in\mathbb{N}$. \label{thm9}
\end{thm}
\begin{proof}
We prove the theorem by using induction. For $k=0$\; we have that $%
C(\omega^{\omega^0}) = C(\omega)\cong c_0(\mathbb{N})$.

In Proposition \ref{prop8}, consider $Y = C[0,1]$,\;$Y_k = <e_k>$,\;$Z_n =
c_0(\mathbb{N})$,\;$Z_{n,k} = \mathbb{R}$, where $(e_k)_{k\in\mathbb{N}}$\;
is a Schauder basis of $C[0,1]$. Since $Y_k, Z_{n,k}$\; are finite
dimensional, the requirements of Proposition \ref{prop8}\; are fulfilled,
thus there exists the desired set $L$.

Suppose that it is true for $k=m\geqslant 0$, it will be shown that it is
true for $k=m+1$.

It is well known that $C(\omega^{\omega^{m+1}}) = \Big(\sum_{k=1}^\infty%
\oplus\big(C(\omega^{\omega^m}),\|\cdot\|_k\big)\Big)_0$, where $\|\cdot\|_k$%
\; is an equivalent norm on $C(\omega^{\omega^m})$. Then the family $%
R_{n,k}Q_n:X\rightarrow C(\omega^{\omega^m})$\; is countable and by the
inductive assumption, there exists a closed, convex, non-dentable subset $%
L^\prime$\; of $K$, such that on $R_{n,k}Q_n(L^\prime)$\; the weak and norm
topologies coincide. Since the PCP and RNP are equivalent on the subsets of $%
K$,\;$L^\prime$\; is non PCP. Applying once more Proposition \ref{prop8}\;
for the set $L^\prime$\; and the family of operators $(Q_n)_{n\in\mathbb{N}}$%
, we conclude that there exists a closed, convex, bounded, non-dentable
subset $L$\; of $L^\prime$, such that on $Q_n(L)$\; the weak and norm
topologies coincide, for all $n\in\mathbb{N}$. The proof is complete.
\end{proof}
\begin{thm}
Let $K$\; be a closed, convex, bounded, non-dentable subset of $%
C(\omega^{\omega^k})$. Then there exists a convex, closed subset $L$\; of $K$%
, such that $L$\; has the PCP and fails the RNP. Therefore the KMP and RNP
are equivalent on the subsets of $C(\omega^{\omega^k})$. \label{thm10}
\end{thm}
\begin{proof}
Towards a contradiction, suppose that $K$ is a closed, convex, bounded
non-dentable subset of $C(\omega ^{\omega ^{k}})$, such that the PCP and RNP
are equivalent on the subsets of $K$. We apply Theorem \ref{thm9}\  for $%
Q=I:C(\omega ^{\omega ^{k}})\rightarrow C(\omega ^{\omega ^{k}})$, the
identity map. Then there exists $L$\  closed, convex, bounded, non-dentable
subset of $K$, such that on $I(L)=L$\  the weak and norm topologies
coincide. But this means that the PCP and RNP are not equivalent on the
subsets of $K$, a contradiction completing the proof.
\end{proof}
\textbf{Problem}. The problem of the equivalence of the Radon Nikodym
Property and the Krein Milman Property, remains open on the subsets of the
spaces $C(\omega ^{\omega ^{a}})$, for ordinals $a\geq \omega $.
\subsection*{Acknowledgements}
We want to express our gratitude to Prof. S.
Argyros for many important discussions and invaluable help.
We also thank the referee for suggestions that greatly simplified the proofs of the theorems.

\end{document}